\documentclass[a4paper,12pt]{article}
\usepackage{amsmath}
\usepackage{amsthm}
\usepackage{amssymb}
\usepackage{mathrsfs}
\usepackage{enumerate}
\usepackage{url}
\usepackage{comment}

\title{The Bou\'e--Dupuis formula 
and the exponential hypercontractivity in the Gaussian space}

\author{Yuu Hariya\thanks{{\it Corresponding author. E-mail:~hariya@tohoku.ac.jp}} \thanks{Mathematical Institute, Tohoku University, Aoba-ku, Sendai 980-8578, Japan. } \and Sou Watanabe\thanks{Yamagata Prefectural Sagae High School, Sagae 991-8511, Japan.}}

\date{\empty}
\numberwithin{equation}{section}

\theoremstyle{plain}

\newtheorem{thm}{Theorem}[section]

\newtheorem{prop}{Proposition}[section]

\newtheorem{cor}{Corollary}[section]

\newtheorem{lem}{Lemma}[section]

\theoremstyle{definition}

\theoremstyle{remark}
\newtheorem{rem}{Remark}[section]

\DeclareMathOperator*{\esssup}{ess\,sup}

\begin{document}

\def\N {\mathbb{N}}
\def\R {\mathbb{R}}
\def\Q {\mathbb{Q}}

\def\calF {\mathcal{F}}

\def\kp {\kappa}

\def\ind {\boldsymbol{1}}

\def\al {\alpha }
\def\la {\lambda }
\def\ve {\varepsilon}
\def\Om {\Omega}

\def\v {v}

\def\ga {\gamma }

\def\W {\mathbb{W}}
\def\H {\mathbb{H}}
\def\A {\mathcal{V}}

\newcommand\ND{\newcommand}
\newcommand\RD{\renewcommand}

\ND\lref[1]{Lemma~\ref{#1}}
\ND\tref[1]{Theorem~\ref{#1}}
\ND\pref[1]{Proposition~\ref{#1}}
\ND\sref[1]{Section~\ref{#1}}
\ND\ssref[1]{Subsection~\ref{#1}}
\ND\aref[1]{Appendix~\ref{#1}}
\ND\rref[1]{Remark~\ref{#1}} 
\ND\cref[1]{Corollary~\ref{#1}}
\ND\eref[1]{Example~\ref{#1}}
\ND\fref[1]{Fig.\ {#1} }
\ND\lsref[1]{Lemmas~\ref{#1}}
\ND\tsref[1]{Theorems~\ref{#1}}
\ND\dref[1]{Definition~\ref{#1}}
\ND\psref[1]{Propositions~\ref{#1}}
\ND\rsref[1]{Remarks~\ref{#1}}
\ND\sssref[1]{Subsections~\ref{#1}}

\ND\pr{\mathbb{P}}
\ND\ex{\mathbb{E}}
\ND\br{B}
\ND\wm{\mathcal{W}}

\ND\be[1]{R^{(#1)}}

\ND\E[1]{\mathcal{E}^{#1}}
\ND\no[2]{\|{#1}\|_{#2}}
\ND\tr[2]{T^{#1}_{#2}}
\ND\si{\mathcal{S}}
\ND\lhs[1]{\log\ex\!\left[e^{{#1}(\br)}\right]}
\ND\inner[2]{({#1},{#2})_{\H }}
\ND\cpl[2]{\langle {#1},{#2}\rangle}
\ND\nt[2]{\left\|{#1}\right\|_{#2}}

\ND\sX{\mathscr{X}}
\ND\sB{\mathscr{B}}
\ND\cP{\mathcal{P}}
\ND\cB{\mathcal{B}}
\ND\cC{\mathcal{C}}
\ND\cS{\mathcal{S}}
\ND\D{d}
\ND\ou{Q}
\ND\gss[1]{\gamma_{#1}}

\def\thefootnote{{}}

\maketitle 

\begin{abstract}
This paper concerns a variational representation formula 
for Wiener functionals. 
Let $\br =\{ \br _{t}\} _{t\ge 0}$ be a standard $d$-dimensional 
Brownian motion. 
Bou\'e and Dupuis (1998) showed that, for any  
bounded measurable functional $F(\br )$ of $\br $ up to time $1$, 
the expectation $\ex \!\left[ e^{F(\br )}\right] $ admits a  
variational representation in terms of drifted Brownian motions. 
In this paper, with a slight modification of insightful 
reasoning by Lehec (2013) allowing also $F(\br )$ to 
be a functional of $\br $ over the whole time interval, we prove that 
the Bou\'e--Dupuis formula 
holds true provided that 
both $e^{F(\br )}$ and $F(\br )$ are integrable, relaxing 
conditions in earlier works. We also show that the formula 
implies the exponential hypercontractivity of the Ornstein--Uhlenbeck 
semigroup in $\R ^{d}$, and hence, 
due to their equivalence, implies the logarithmic Sobolev 
inequality in the $d$-dimensional Gaussian space.
\end{abstract}

\section{Introduction}\label{;intro}
Given a positive integer $d$, let $\br =\{ \br _{t}\} _{t\ge 0}$ be 
a standard $d$-dimensional Brownian motion. 
In \cite{bd}, Bou\'e and Dupuis established the following formula  
for any bounded measurable function $F$  
mapping $C([0,1];\R ^{d})$ into $\R $: 
\begin{align}\label{;vr0}
 \log \ex \!\left[ 
 e^{F(\br )}
 \right] 
 =\sup _{\v }\ex \!\left[ 
 F\left( \br +\int _{0}^{\cdot }\v _{t}\,dt\right) 
 -\frac{1}{2}\int _{0}^{1}|\v _{t}|^{2}\,dt
 \right] . 
\end{align}
Here the supremum runs over all progressively measurable 
processes $\v$ with respect to the augmented natural 
filtration of $\br $ such that $\int _{0}^{1}|\v _{t}|^{2}\,dt$ 
is integrable. In \cite{bd}, formula \eqref{;vr0} 
was proven useful 
in deriving various large deviation asymptotics such as Laplace 
principles for small noise diffusions driven by Brownian motion. 
These results have been extended by 
Budhiraja and Dupuis \cite{bud} to Hilbert space-valued 
Brownian motions, and later generalized by Zhang \cite{zha} to 
abstract Wiener spaces. In Bou\'e--Dupuis \cite{bd2}, 
formula \eqref{;vr0} is also applied to risk-sensitive stochastic 
control problems. Recently, the formula has been used 
effectively by Barashkov and Gubinelli \cite{bg} in the study of the 
$\Phi ^{4}_{3}$ Gibbs measure in the quantum field theory (in fact, 
they employ an extended formula by \"Ust\"unel \cite{ust} to a 
class of unbounded functions $F$; see \rref{;rtmain1}\thetag{3}).
Their idea is exploited further by Chandra, Gunaratnam 
and Weber in \cite{cgw}.

One of the objectives of this paper is to show that 
the boundedness imposed on the functions $F$ is 
removable when both $e^{F(\br )}$ and $F(\br )$ are integrable. 
We do this by slightly modifying reasoning by Lehec \cite{leh} 
based on deep understanding of the Gaussian relative 
entropy, which also allows $F(\br )$ to be a functional of 
$\br $ over the whole time interval. 

In order to state the result precisely, we prepare some of notation. 
We denote by $\pr $ the probability measure of the 
probability space on which the Brownian motion $\br $ is 
defined. We set 
\begin{align*}
 \calF ^{\br }_{t}:=\sigma (\br _{s},0\le s\le t)\vee \mathcal{N}, \quad 
 t\ge 0, 
\end{align*}
the filtration generated by $\br $ and augmented by the set 
$\mathcal{N}$ of all $\pr $-null events. Let $\v =\{ \v _{t}\} _{t\ge 0}$ 
be a $d$-dimensional process defined on the same probability 
space as $\br $. We call $\v $ a {\it drift} if it is 
$\{ \calF ^{\br }_{t}\} $-progressively measurable. We denote by 
$\A $ the set of drifts $\v $ satisfying 
\begin{align}\label{;intv}
 \ex \!\left[ \int _{0}^{\infty }|\v _{t}|^{2}\,dt\right] <\infty . 
\end{align}
Here and in what follows, $\ex $ denotes the expectation with 
respect to $\pr $ and $|x|$ stands for the Euclidean norm of 
$x\in \R ^{d}$. 

Let $\W =C([0,\infty );\R ^{d})$ be the space 
of $\R ^{d}$-valued continuous functions on $[0,\infty )$ vanishing 
at the origin, endowed with the topology of uniform convergence 
on compact subsets of $[0,\infty )$. 
We denote by $\mathcal{B}(\W )$ the associated Borel 
$\sigma $-field and by $\wm $ the Wiener measure 
on $(\W ,\mathcal{B}(\W ))$. 
Let $F:\W \to \R $ be measurable. We assume: 
\begin{align*}
 \thetag{A1}\ \int _{\W }e^{F}\,d\wm <\infty ; && 
 \thetag{A2}\ \int _{\W }F_{-}\,d\wm <\infty ,
\end{align*}
where we set $F_{-}(w):=\max \{ -F(w),0\} ,\,w\in \W $. 
The main result of the paper is stated as 
follows: 

\begin{thm}\label{;tmain1}
 Let $F:\W \to \R $ be a measurable function satisfying 
 \thetag{A1} and \thetag{A2}. Then it holds that 
 \begin{align}\label{;vr1}
  \log \ex \!\left[ 
 e^{F(\br )}
 \right] 
 =\sup _{\v \in \A }\ex \!\left[ 
 F\left( \br +\int _{0}^{\cdot }\v _{t}\,dt\right) 
 -\frac{1}{2}\int _{0}^{\infty }|\v _{t}|^{2}\,dt
 \right] . 
 \end{align}
\end{thm}

We may replace the supremum over $\A $ by that over a 
class of bounded drifts; see \cref{;ctmain1}.

We give a remark on \tref{;tmain1}. 

\begin{rem}\label{;rtmain1}
\thetag{1}~Under assumption 
\thetag{A1},  the right-hand side of \eqref{;vr1} is 
well-defined in the sense that, for any $\v \in \A $, 
\begin{align*}
 \ex \!\left[ 
 F_{+}\left( \br +\int _{0}^{\cdot }\v _{t}\,dt\right) 
 \right] <\infty  \quad \text{with} \quad 
 F_{+}:=\max \left\{ F,0\right\} , 
\end{align*}
while 
$
 \ex \!\left[ 
 F_{-}\left( \br +\int _{0}^{\cdot }\v _{t}\,dt\right) 
 \right] 
$ 
may take value $\infty $ for some $\v \in \A $; 
see the beginning of the proof of \pref{;plb}. 

\noindent 
\thetag{2}~Although a little involved argument is used in 
\cite[Section~5]{bd}, the extension of formula~\eqref{;vr0} 
to the case that $F$ is only assumed to be bounded from 
below is immediate from the monotone convergence theorem. 
Indeed, for each positive real $M$, truncating $F$ from above by 
$M$, we have from \eqref{;vr0}, 
\begin{align*}
 \log \ex \!\left[ 
 e^{F_{M}(\br )}
 \right] 
 =\sup _{\v \in \A }\ex \!\left[ 
 F_{M}\left( \br +\int _{0}^{\cdot }\v _{t}\,dt\right) 
 -\frac{1}{2}\int _{0}^{1}|\v _{t}|^{2}\,dt
 \right] , 
\end{align*}
where  
$
F_{M}:=
\min \left\{ F,M\right\} 
$; then, by the monotone convergence theorem, 
the left-hand side converges as $M\to \infty $ to 
the expression with $F_{M}$ replaced by $F$, and so does 
the right-hand side since 
\begin{equation*}
 \begin{split}
 &\sup _{M>0}\sup _{\v \in \A }
 \ex \!\left[ 
 F_{M}\left( \br +\int _{0}^{\cdot }\v _{t}\,dt\right) 
 -\frac{1}{2}\int _{0}^{1}|\v _{t}|^{2}\,dt
 \right] 
 \\
 =&\sup _{\v \in \A }\sup _{M>0}
 \ex \!\left[ 
 F_{M}\left( \br +\int _{0}^{\cdot }\v _{t}\,dt\right) 
 -\frac{1}{2}\int _{0}^{1}|\v _{t}|^{2}\,dt
 \right] \\
 =&\sup _{\v \in \A }\ex \!\left[ 
 F\left( \br +\int _{0}^{\cdot }\v _{t}\,dt\right) 
 -\frac{1}{2}\int _{0}^{1}|\v _{t}|^{2}\,dt
 \right] . 
 \end{split}
\end{equation*}
In this respect, what is essential in \tref{;tmain1} is the 
removal of the boundedness of $F$ from below. 

\noindent 
\thetag{3}~In \cite[Theorem~7]{ust}, formula~\eqref{;vr0} 
is proven under the condition that, for some 
$p,q>1$ with $p^{-1}+q^{-1}=1$, 
\begin{align*}
 \ex \!\left[ 
 \left|F(\br )\right| ^{p}\right] <\infty \quad \text{and} \quad 
 \ex \!\left[ e^{qF(\br )}\right] <\infty ,
\end{align*}
while, in \cite[Theorem~1.1]{har}, the condition that 
\begin{align*}
 \ex \!\left[ 
 \left|F(\br )\right| ^{p}\right] <\infty \ \text{for some }p>1 
 \quad \text{and} \quad 
 \ex \!\left[ e^{F(\br )}\right] <\infty 
\end{align*}
is imposed. Due to their methods, the restriction $p>1$ seems 
inevitable: \cite{ust} uses its Lemma~1 and \cite{har} its Lemma~2.10.
Note that our assumption of \tref{;tmain1} is 
equivalently rephrased as 
\begin{align*}
 \ex \!\left[ 
 \left|F(\br )\right| \right] <\infty \quad \text{and} \quad 
 \ex \!\left[ e^{F(\br )}\right] <\infty .
\end{align*}
\end{rem}

Independently of the work \cite{bd} by Bou\'e--Dupuis, 
Borell \cite{bor} proved formula~\eqref{;vr0} when $F(\br )$ 
is of the form $f(\br _{1})$ with 
$f:\R ^{d}\to \R $ a bounded measurable function, and, among 
other applications, applied it to a simple derivation of 
the Pr\'ekopa--Leindler inequality. 
In the last part of the paper, we will show that it also 
yields readily the exponential version of the hypercontractivity of 
the Ornstein--Uhlenbeck semigroup in $\R ^{d}$; 
the equivalence between the exponential 
hypercontractivity and the logarithmic Sobolev inequality then entails 
that the formula implies the Gaussian logarithmic Sobolev inequality.

\medskip 
We give an outline of the paper. In \sref{;prfvr}, we prove 
\tref{;tmain1}; the lower bound in formula~\eqref{;vr1} is 
proven in \ssref{;ssplb} while the upper bound is 
proven in \ssref{;sspub}, where 
the case of bounded drifts is also stated in \cref{;ctmain1}. 
The paper is concluded with \sref{;sappl} that explores 
the above-mentioned connection between 
the formula and the exponential hypercontractivity of the 
Ornstein--Uhlenbeck semigroup.

\medskip 
For every $a,b\in \R $, we write $a\vee b=\max\{ a,b\} $ and 
$a\wedge b=\min \{ a,b\} $. 
For a positive integer $n$, we denote by 
$C_{b}^{\infty }(\R ^{n})$ the set of real-valued bounded 
$C^{\infty }$-functions on $\R ^{n}$ whose partial 
derivatives are all bounded. 
Given a measured space $(\sX ,\sB ,m)$, for every 
$1\le p\le \infty $, 
we denote by $L^{p}(m)$ the set of real-valued 
measurable functions $f$ on the measurable space $(\sX ,\sB )$ 
such that 
\begin{align*}
 \bigl\{ \nt{f}{L^{p}(m)}\bigr\} ^{p}:=\int _{\sX }|f(x)|^{p}\,m(dx)<\infty && 
 \text{for $p<\infty $,}
\intertext{and that, with $\esssup $ denoting the essential supremum 
with respect to $m$,}
 \nt{f}{L^{\infty }(m)}:=\esssup _{x\in \sX }|f(x)| <\infty && 
 \text{for $p=\infty $.}
\end{align*}
Other notation will be introduced as needed. 

\section{Proof of \tref{;tmain1}}\label{;prfvr} 

This section is devoted to the proof of \tref{;tmain1}. 

Let $(\sX ,\sB )$ be a measurable space and 
$\cP (\sX )$ the set of probability measures on it. 
For $\mu ,\nu \in \cP (\sX )$, recall that the relative 
entropy $H(\nu \mid \mu )$ of $\nu $ with respect to 
$\mu $ is defined by 
\begin{align*}
 H(\nu \mid \mu ):=
 \begin{cases}
 \displaystyle 
 \int _{\sX }\frac{d\nu }{d\mu }\log \frac{d\nu }{d\mu }\,d\mu & 
 \text{if $\nu \ll \mu $},\\
 +\infty & \text{otherwise}
 \end{cases}
\end{align*}
(see, e.g., \cite[Section~1.4]{de}).

In the sequel, for every drift $\v $, we denote 
\begin{align*} 
 \br ^{\v }_{t}=\br _{t}+\int _{0}^{t}\v _{s}\,ds, 
 \quad t\ge 0, 
\end{align*}
the Brownian motion drifted by $\v$ and, whenever $\v \in \A $, 
\begin{align*}
 \nt{\v }{\A }=
 \left\{ \ex \!\left[ 
 \int _{0}^{\infty }|\v _{t}|^{2}\,dt
 \right] \right\} ^{1/2}. 
\end{align*}
A drift $\v $ is said to be {\it bounded} if it satisfies 
\begin{align*}
 \sup _{t\ge 0}\nt{|\v _{t}|}{L^{\infty }(\pr )}<\infty .
\end{align*}
For later use in \ssref{;sspub}, we set
\begin{align*}
 \A _{b}&:=\left\{ 
 \v ;\,\text{$\v $ is a bounded drift satisfying \eqref{;intv}}
 \right\} ,\\
 \A _{b,0}&:=\left\{ 
 \v ;\begin{array}{l}
 \text{$\v $ is a bounded drift satisfying that there exists}\\
 \text{$K>0$ such that $\v _{t}=0$ a.s.\ for all $t\ge K$}
 \end{array}
 \right\} .
\end{align*}
Clearly, we have $\A _{b,0}\subset \A _{b}\subset \A $.

\subsection{Lower bound}\label{;ssplb}
In this subsection, we give a proof of the lower bound in 
\eqref{;vr1}, that is, we prove 

\begin{prop}\label{;plb}
Assume that a measurable function $F:\W \to \R $ satisfies 
\thetag{A1}. Then we have
\begin{align}\label{;vr1l}
\log \ex \!\left[ 
 e^{F(\br )}
 \right] 
 \ge \sup _{\v \in \A }\left\{ 
 \ex \!\left[ F\!\left( 
 \br ^{\v }
 \right) 
 \right] -\frac{1}{2}\nt{\v }{\A }^{2}
 \right\} .
\end{align}
\end{prop}

\pref{;plb} is immediate once the following 
lemma is at our disposal.

\begin{lem}\label{;llb}
Under \thetag{A1}, the lower bound \eqref{;vr1l} holds 
when $F$ is bounded from below.
\end{lem}

By using this lemma, \pref{;plb} is proven as follows: 
\begin{proof}[Proof of \pref{;plb}]
 First we verify that under assumption \thetag{A1}, 
 \begin{align}\label{;fp}
  \ex \!\left[ 
  F_{+}\!\left( \br ^{\v }\right) 
  \right] <\infty \quad 
  \text{for any }\v \in \A , 
 \end{align}
 where $F_{+}(w):=F(w)\vee 0,\,w\in \W $. 
 Fix $\v \in \A $ arbitrarily and set  
 $F_{+,M}=F_{+}\wedge M$ for each $M>0$. 
 Then, by \lref{;llb}, we have in particular 
 \begin{align*}
  \ex \!\left[ 
  F_{+,M}\!\left( \br ^{\v }\right) 
  \right] \le \log \ex \!\left[ 
  e^{F_{+,M}(\br )}\right] +\frac{1}{2}\nt{\v }{\A }^{2}. 
 \end{align*}
 By letting $M\to \infty $, the monotone convergence theorem 
 entails that 
 \begin{align*}
  \ex \!\left[ 
  F_{+}\!\left( \br ^{\v }\right) 
  \right] &\le \log \ex \!\left[ 
  e^{F_{+}(\br )}\right] +\frac{1}{2}\nt{\v }{\A }^{2}\\
  &\le \log \ex \!\left[ 
  1+e^{F_{}(\br )}\right] +\frac{1}{2}\nt{\v }{\A }^{2}, 
 \end{align*}
 which is finite by \thetag{A1}. 
 
 For every $N>0$, we now define 
 \begin{align*}
  F_{N}(w):=F(w)\vee (-N),\quad w\in \W . 
 \end{align*}
 Then, by \lref{;llb}, the lower bound \eqref{;vr1l} 
 holds for $F_{N}$:  
 \begin{align}\label{;vr1M}
  \lhs{F_{N}}\ge 
  \sup _{\v \in \A }
  \left\{ 
  \ex \!\left[ 
  F_{N}\!\left( 
  \br ^{\v }
  \right) 
  \right] 
  -\frac{1}{2}\nt{\v }{\A }^{2}
  \right\} .
 \end{align}
 By assumption \thetag{A1}, 
 the random variable $\sup \limits_{N>0}e^{F_{N}(\br )}$ is integrable 
 and so is $\sup \limits_{N>0}F_{N}\!\left( \br ^{\v }\right) $ 
 for any $\v \in \A $ thanks to \eqref{;fp}. 
 Therefore, as $N\to \infty $, we may use the monotone 
 convergence theorem on both sides 
 of \eqref{;vr1M} to get 
 \begin{align*}
  \lhs{F}&\ge \inf _{N>0}\sup _{\v \in \A }
  \left\{ 
  \ex \!\left[ 
  F_{N}\!\left( 
  \br ^{\v }
  \right) 
  \right] 
  -\frac{1}{2}\nt{\v }{\A }^{2}
  \right\} \\
  &\ge \sup _{\v \in \A }\inf _{N>0}
  \left\{ 
  \ex \!\left[ 
  F_{N}\!\left( 
  \br ^{\v }
  \right) 
  \right] 
  -\frac{1}{2}\nt{\v }{\A }^{2}
  \right\} \\
  &=\sup _{\v \in \A }
  \left\{ 
  \ex \!\left[ 
  F_{}\!\left( 
  \br ^{\v }
  \right) 
  \right] 
  -\frac{1}{2}\nt{\v }{\A }^{2}
  \right\} , 
 \end{align*}
 which is the assertion. 
\end{proof}

We proceed to the proof of \lref{;llb}. We prepare 
two lemmas, the first one of which is adapted from 
\cite[Proposition~4.5.1]{de}.
\begin{lem}\label{;lde}
 Let $F:\W \to \R $ be a measurable function bounded from 
 below. Then it holds that 
 \begin{align*}
  \log \int _{\W }e^{F}\,d\wm 
  =\sup _{\mu \in \Delta (\W )}\left\{ 
  \int _{\W }F\,d\mu -H(\mu \mid \wm )
  \right\} ,
 \end{align*}
 where 
 $\Delta (\W ):=\left\{ 
  \mu \in \cP (\W );\,H(\mu \mid \wm )<\infty 
  \right\} 
 $.
\end{lem}

The second one is taken from \cite{leh}.
\begin{lem}[\cite{leh}, Proposition~1]\label{;lleh1}
 Let $\v $ be a drift and $\mu $ the law of 
 $\br ^{\v }$. 
 Then it holds that 
 \begin{align*}
  H(\mu \mid \wm )
  \le \frac{1}{2}\ex \!\left[ \int _{0}^{\infty }|\v _{t}|^{2}\,dt \right] .
 \end{align*}
\end{lem}

Combining these lemmas yields \lref{;llb} readily. 
\begin{proof}[Proof of \lref{;llb}]
For an arbitrary $\v \in \A $, let $\mu $ be the law of $\br ^{\v }$. 
Then, since $\mu \in \Delta (\W )$ by the definition of 
$\A $ and \lref{;lleh1}, we have from 
\lref{;lde} that 
\begin{align*}
 \log \int _{\W }e^{F}\,d\wm &\ge 
 \int _{\W }F\,d\mu -H(\mu \mid \wm )\\
 &\ge \int _{\W }F\,d\mu -\frac{1}{2}\nt{\v }{\A }^{2}, 
\end{align*}
where we used \lref{;lleh1} again for the second line. 
The assertion is proven because $\mu $ is the law of 
$\br ^{\v }$ and $\v $ is arbitrary.
\end{proof}

\subsection{Upper bound}\label{;sspub} 
In this subsection, we prove the upper bound in 
\eqref{;vr1}:  

\begin{prop}\label{;pub}
Assume that a measurable function $F:\W \to \R $ satisfies 
\thetag{A1} and \thetag{A2}. Then we have 
\begin{align}\label{;eqpub}
 \log \ex \!\left[ 
 e^{F(\br )}
 \right] 
 \le \sup _{\v \in \A }\left\{ 
 \ex \!\left[ F\!\left( 
 \br ^{\v }
 \right) 
 \right] -\frac{1}{2}\nt{\v }{\A }^{2}
 \right\} .
\end{align}
\end{prop}

We denote by $\calF C_{b}^{\infty }$ the set of functions $\Phi $ 
on $\W $ of the form 
\begin{align}\label{;fcbi}
 \Phi (w)=\phi \left( w(t_{1}),\ldots ,w(t_{m})\right) ,\quad w\in \W ,  
\end{align}
for some $m\in \N $, $0\le t_{1}<\cdots <t_{m}$ and 
$\phi \in C_{b}^{\infty }(\R ^{d\times m})$. 
We also denote by $\mathcal{C}$ the set of cylinder subsets $C$ of 
$\W $, namely, each $C$ is of the form 
\begin{align}\label{;cylinder}
 C=\left\{ 
 w\in \W ;\,\left( w(t_{1}),\ldots ,w(t_{m})\right) \in \Gamma 
 \right\} 
\end{align}
for some $m\in \N $ and $0\le t_{1}<\cdots <t_{m}$, and for some 
Borel subset $\Gamma $ of $\R ^{d\times m}$. It is well known that 
\begin{align}\label{;sigma}
 \sigma (\mathcal{C})=\mathcal{B}(\W )
\end{align} 
(see, e.g., \cite[Problem~2.4.2]{ks}).
Let $F:\W \to \R $ be a measurable function 
and define the $\sigma $-finite measure $\nu _{F}$ on 
$(\W ,\mathcal{B}(\W ))$ by 
\begin{align*}
 d\nu _{F}:=(1+F_{-})\,d\wm .
\end{align*}
If $F$ fulfills \thetag{A2}, then $\nu _{F}$ is a finite measure 
and the following lemma is standard but crucial to our 
argument.
\begin{lem}\label{;ldense}
 $\calF C_{b}^{\infty }$ is dense in $L^{2}(\nu _{F})$ 
 under assumption \thetag{A2}.
\end{lem}

For the completeness of the paper, we give a proof.

\begin{proof}[Proof of \lref{;ldense}]
It suffices to show that, for any $A\in \cB (\W )$, its indicator 
function $\ind _{A}$ can be approximated by a sequence 
$\{ \Phi _{n}\} _{n=1}^{\infty }$ in 
$\calF C_{b}^{\infty }$ in $L^{2}(\nu _{F})$. To this end, fix a 
positive integer $n$ arbitrarily. In view of 
\eqref{;sigma}, the approximation property 
(e.g., \cite[Theorem~5.7]{bau}) entails that 
there exists $C_{}\in \cC $ such that 
\begin{align*}
 \nt{\ind _{A}-\ind _{C_{}}}{L^{2}(\nu _{F})}<n^{-1} 
\end{align*}
because of the fact that $\nu _{F}$ is a finite measure and 
$
 |\ind _{A}-\ind _{C_{}}|^{2}=\ind _{A\Delta C_{}}
$, 
where the symbol $\Delta $ stands for the symmetric difference. 
As $C_{}$ may be expressed as 
\eqref{;cylinder}, it is now routine to check that elements 
of $C_{b}^{\infty }(\R ^{d\times m})$ approximate $\ind _{\Gamma }$ 
in the sense of $L^{2}$ under the finite measure 
\begin{align*}
 \nu _{F}^{m}(\,\cdot \,)
 \equiv \nu _{F}^{t_{1},\ldots ,t_{m}}(\,\cdot \,):=\nu _{F}\!\left( 
 \left\{ w\in \W ;\,
 \left( w(t_{1}),\ldots ,w(t_{m})\right) \in \,\cdot \,
 \right\} 
 \right) 
\end{align*}
on $\R ^{d\times m}$. To see that, notice that $\nu _{F}^{m}$ is 
inner regular (cf.\ \cite[Lemma~26.2]{bau}). Hence there exists a 
compact subset $K$ of $\Gamma $ such that 
\begin{align*}
 \nt{\ind _{\Gamma }-\ind _{K}}{L^{2}(\nu _{F}^{m})}
 =\sqrt{\nu _{F}^{m}(\Gamma )-\nu _{F}^{m}(K)}<n^{-1}.
\end{align*}
Convoluting $\ind _{K}$ with the standard mollifier 
(\cite[Subsection~4.2.1]{eg}), we may construct a family 
$\{ \phi _{\ve }\} _{\ve >0}\subset C_{b}^{\infty }(\R ^{d\times m})$ 
(in fact, each $\phi _{\ve }$ is compactly supported) such that  
\begin{align*}
 \phi _{\ve }\to \ind _{K}\quad \text{a.e.\ as $\ve \downarrow 0$}. 
\end{align*}
Thanks to the finiteness of $\nu _{F}^{m}$, 
the above a.e.\ convergence entails that there exists 
$\phi \in C_{b}^{\infty }(\R ^{d\times m})$ such that 
\begin{align*}
 \nt{\ind _{K}-\phi }{L^{2}(\nu _{F}^{m})}<n^{-1}
\end{align*}
by the bounded convergence theorem. Therefore, setting 
\begin{align*}
 \Phi _{n}(w):=
 \phi \left( w(t_{1}),\ldots ,w(t_{m})\right) ,\quad w\in \W ,
\end{align*}
we have the desired sequence 
$\{ \Phi _{n}\} _{n=1}^{\infty }\subset \calF C_{b}^{\infty }$ 
because 
\begin{align*}
 \nt{\ind _{A}-\Phi _{n}}{L^{2}(\nu _{F})}<3n^{-1}
\end{align*}
for each $n$ by construction.
\end{proof}

Following the notation of \cite{leh}, we define 
\begin{align*}
 \cS :=\left\{ 
 \mu \in \cP (\W );\,
 \text{$\mu $ has a density $\Phi \in \calF C_{b}^{\infty }$ 
 w.r.t.\ $\wm $ such that $\inf _{w\in \W }\Phi (w)>0$}
 \right\} .
\end{align*}
The next lemma is also adapted from \cite{leh}. 

\begin{lem}[\cite{leh}, Theorem~7]\label{;lleh2}
 For every $\mu \in \cS $, there exists $\v \in \A $ such that 
 $\br ^{\v }$ has law $\mu $ and 
 \begin{align}\label{;min}
  H(\mu \mid \wm )=\frac{1}{2}\nt{\v }{\A }^{2}.  
 \end{align}
\end{lem}

\begin{rem}\label{;rmin}
With $u:[0,\infty )\times \W \to \W $ the {\it F\"ollmer process} 
associated with $\mu $, as constructed in the proof of 
\cite[Theorem~2]{leh}, one of $\v $'s fulfilling \eqref{;min} is given 
by $\v =\{ u(t,X)\} _{t\ge 0}$, where $X=\{ X_{t}\} _{t\ge 0}$ is the 
unique strong solution to the stochastic differential equation 
\begin{align*}
 dX_{t}=d\br _{t}+u(t,X)\,dt,\quad t\ge 0,\ X_{0}=0.
\end{align*}
The above choice of $\v $ is in $\A _{b,0}$; 
indeed, supposing that 
$\mu \in \cS $ has density $\Phi $ given by \eqref{;fcbi}, 
we see that 
\begin{align*}
 |\v _{t}|\le 
 \frac{1}{\inf \limits_{x\in \R ^{d\times m}}\phi (x)}
 \sum _{i=1}^{m}\sup _{x\in \R ^{d\times m}}\left| 
 \nabla _{x^{i}}\phi (x)
 \right| \quad \text{a.s.}
\end{align*}
for $0\le t\le t_{m}$ and $\v _{t}=0$ for $t>t_{m}$ 
by construction. Here, for each 
$1\le i\le m$, $\nabla _{x^{i}}\phi $ is the gradient of 
$\phi (x)\equiv \phi (x^{1},\ldots ,x^{m})$ with respect to 
the variable $x^{i}\in \R ^{d}$.
\end{rem}

Combining \lsref{;ldense} and \ref{;lleh2}, we immediately obtain 
\begin{prop}\label{;pubba}
 The upper bound \eqref{;eqpub} holds for any measurable 
 function $F:\W \to \R $ that is bounded from above and 
 satisfies \thetag{A2}. 
\end{prop}

\begin{proof}
Set $G:=e^{F}$. Without loss of generality, we may assume 
$\nt{G}{L^{1}(\wm )}=1$. As $G\in L^{2}(\nu _{F})$ thanks to the 
boundedness of $G$, there exists a sequence 
$\{ \Phi _{n}\} _{n=1}^{\infty }\subset \calF C_{b}^{\infty }$ 
such that 
\begin{align}\label{;l2}
 \lim _{n\to \infty }\nt{\Phi _{n}-G}{L^{2}(\nu _{F})}=0
\end{align}
by \lref{;ldense}. For every $n$, truncating $\Phi _{n}$ if necessary, 
we may assume $\inf \limits_{w\in \W }\Phi _{n}(w)>0$. 
For each $n$, define $G_{n}:=\Phi _{n}/\nt{\Phi _{n}}{L^{1}(\wm )}$ 
so that $d\mu _{n}:=G_{n}\,d\wm $ is in $\cS $. It is clear that 
\begin{align}\label{;l2d}
 \lim _{n\to \infty }\nt{G_{n}-G}{L^{2}(\nu _{F})}=0
\end{align}
by \eqref{;l2}; indeed, 
\begin{align*}
 \nt{G_{n}-G}{L^{2}(\nu _{F})}&\le 
 \frac{1}{\nt{\Phi _{n}}{L^{1}(\wm )}}
 \nt{\Phi _{n}-G}{L^{2}(\nu _{F})}
 +\left| 
 \frac{1}{\nt{\Phi _{n}}{L^{1}(\wm )}}-1
 \right| \nt{G}{L^{2}(\nu _{F})}, 
\end{align*}
which tends to $0$ because \eqref{;l2} also entails that 
$
\lim \limits_{n\to \infty }\nt{\Phi _{n}}{L^{1}(\wm )}
=\nt{G}{L^{1}(\wm )}=1
$. As $\{ G_{n}\} _{n=1}^{\infty }$ is bounded in 
$L^{2}(\wm )$ by \eqref{;l2d} and the definition of $\nu _{F}$, 
the sequence $\{ G_{n}\log G_{n}\} _{n=1}^{\infty }$ is 
uniformly integrable under $\wm $, whence, by Vitali's convergence 
theorem (see, e.g., \cite[Theorem~22.7]{sch}), 
\begin{align}\label{;lim1}
 \lim _{n\to \infty }\int _{\W }G_{n}\log G_{n}\,d\wm 
 =\int _{\W }G_{}\log G_{}\,d\wm 
\end{align}
because \eqref{;l2d} also implies $G_{n}\to G$ in probability 
under $\wm $. Moreover, it follows that 
\begin{align*}
 \lim _{n\to \infty }\int _{\W }G_{n}F_{-}\,d\wm 
 =\int _{\W }G_{}F_{-}\,d\wm .
\end{align*}
Since $\{ G_{n}\} _{n=1}^{\infty }$ also converges to $G$ in 
$L^{1}(\wm )$ and $F_{+}$ is bounded, we have 
\begin{align*}
 \lim _{n\to \infty }\int _{\W }G_{n}F_{+}\,d\wm 
 =\int _{\W }G_{}F_{+}\,d\wm 
\end{align*}
as well, and hence 
\begin{align}\label{;lim2}
 \lim _{n\to \infty }\int _{\W }FG_{n}\,d\wm =\int _{\W }FG\,d\wm .
\end{align}
Combining \eqref{;lim1} and 
\eqref{;lim2}, we see that 
\begin{align*}
 \int _{\W }F\,d\mu _{n}-H(\mu _{n}\mid \wm )
 &=\int _{\W }FG_{n}\,d\wm -\int _{\W }G_{n}\log G_{n}\,d\wm \\
 &\xrightarrow[n\to \infty ]{}0
\end{align*}
by the definition of $G$. 
Therefore, for any $\ve >0$, there exists $\mu \in \cS $ such that, 
because of $\ex \!\left[ e^{F(\br )}\right] =1$,  
\begin{align*}
 \log \ex \!\left[ e^{F(\br )}\right] 
 <\int _{\W }F\,d\mu -H(\mu \mid \wm )+\ve .
\end{align*}
The right-hand side is dominated by 
\begin{align}\label{;dom}
 \sup _{\v \in \A }\left\{ 
 \ex \!\left[ F\!\left( 
 \br ^{\v }
 \right) 
 \right] -\frac{1}{2}\nt{\v }{\A }^{2}
 \right\} +\ve 
\end{align}
in view of \lref{;lleh2}, which proves the proposition 
as $\ve >0$ is arbitrary. 
\end{proof}

\begin{rem}
If we let $\v _{n}\in \A $ be as in \rref{;rmin} for each $\mu _{n}$, 
what is in fact proven is 
\begin{align*}
 \log \ex \!\left[ e^{F(\br )}\right] 
 =\lim _{n\to \infty }\left\{ 
 \ex \!\left[ F(\br ^{\v _{n}})\right] -\frac{1}{2}\nt{\v _{n}}{\A }^{2}
 \right\} .
\end{align*}
\end{rem}

We finish the proof of \pref{;pub}.  
\begin{proof}[Proof of \pref{;pub}]
 For a measurable function $F:\W \to \R $ satisfying \thetag{A1} and 
 \thetag{A2}, we set for each $M>0$, 
 \begin{align*}
  F_{M}(w):=F(w)\wedge M, \quad w\in \W . 
 \end{align*}
 Then, for any $M$, we have by \pref{;pubba}, 
 \begin{align*}
  \lhs{F_{M}}&\le 
  \sup _{\v \in \A }\left\{ 
  \ex \!\left[ F_{M}\!\left( 
  \br ^{\v }
  \right) 
  \right] -\frac{1}{2}\nt{\v }{\A }^{2}
  \right\} \\
  &\le \sup _{\v \in \A }\left\{ 
 \ex \!\left[ F\!\left( 
 \br ^{\v }
 \right) 
 \right] -\frac{1}{2}\nt{\v }{\A }^{2}
 \right\} , 
 \end{align*}
the last expression being well-defined by \eqref{;fp}. 
Letting $M\to \infty $ on the leftmost side completes the proof 
by the dominated/monotone convergence theorem. 
\end{proof}

Since domination \eqref{;dom} is valid 
if we replace the supremum over $\A $ by that over $\A _{b}$ 
or $\A _{b,0}$ in view of \rref{;rmin}, 
we have the following corollary, which we think is 
useful in some of applications; see, e.g., 
\cite[Remarks~4.8 and 4.9]{cgw}.

\begin{cor}\label{;ctmain1}
The supremum in \eqref{;vr1} may be replaced by that over 
drifts $\v $ in $\A _{b}$ or $\A _{b,0}$; that is, for any measurable 
function $F:\W \to \R $ satisfying \thetag{A1} and \thetag{A2}, 
we have 
\begin{align*}
  \log \ex \!\left[ 
 e^{F(\br )}
 \right] 
 &=\sup _{\v \in \A _{b}}\left\{ 
 \ex \!\left[ F\!\left( 
 \br ^{\v }
 \right) 
 \right] -\frac{1}{2}\nt{\v }{\A }^{2}
 \right\} \\
 &=\sup _{\v \in \A _{b,0}}\left\{ 
 \ex \!\left[ F\!\left( 
 \br ^{\v }
 \right) 
 \right] -\frac{1}{2}\nt{\v }{\A }^{2}
 \right\} . 
 \end{align*}
\end{cor}

We end this section with a remark on the proof of \tref{;tmain1} and 
related facts. 
\begin{rem}\label{;rptmain1}
\thetag{1}\ Since both sides of \eqref{;vr1} are well-defined
only under assumption \thetag{A1} as noted in 
\rref{;rtmain1}\thetag{1}, it is plausible that formula 
\eqref{;vr1} holds true without any assumptions on 
$F$ from below; however, we have not succeeded in proving it. 
The difficulty is to prove the upper bound \eqref{;eqpub} 
without assuming \thetag{A2}.

\noindent 
\thetag{2} Using the notion of filtrations introduced by 
\"Ust\"unel and Zakai \cite{uz} on abstract Wiener spaces, 
Zhang \cite{zha} extended formula~\eqref{;vr0} of 
Bou\'e--Dupuis for bounded Wiener functionals to the 
framework of abstract Wiener spaces as simplifying 
the original proof of the upper bound which relied on a 
complicated measurable selection argument. 
As for the case of the Wiener space 
$(\W ,\wm )$, Lehec \cite{leh} further simplified 
the proof of the upper bound, based on deep 
analysis of the Gaussian relative entropy 
as exhibited in \lsref{;lleh1} and \ref{;lleh2}.  
Note that Lehec's extension \cite[Theorem~9]{leh} to the 
case with $F(\br )$ a functional of $\br $, assumed 
bounded from below, over the whole time 
interval may be seen as a particular case of 
Zhang's result \cite[Theorem~3.2]{zha}; indeed, as 
discussed in \cite[Section~8.1]{str}, by restricting 
$\wm $ to the Banach space $\tilde{\W }$ consisting of paths 
$w \in \W $ such that $\lim \limits_{t\to \infty }|w(t)|/t=0$ 
normed by 
$\sup \limits_{t\ge 0}|w(t)|/(1+t)$, the triple $(\tilde{\W },\H ,\wm )$ 
forms an abstract 
Wiener space, where $\H $ is the usual Cameron--Martin subspace 
of $\W $. 

\noindent 
\thetag{3} One of the main differences between Lehec's proof and 
ours is that we appeal to the density of $\calF C_{b}^{\infty }$ in 
$L^{2}(\nu _{F})$ instead of $L^{2}(\wm )$; another is the employment 
of Vitali's convergence theorem in \eqref{;lim1}.
\end{rem}

\section{Application to the Ornstein--Uhlenbeck semigroup}\label{;sappl}

In this section, we explore a connection between 
formula~\eqref{;vr1} and the exponential version of the 
hypercontractivity of the Ornstein--Uhlenbeck semigroup 
in $\R ^{d}$. For this purpose, we begin with restating \tref{;tmain1} 
when the functional $F(\br )$ is a function of $\br _{1}$. 

We consider the set of $d$-dimensional 
$\{ \calF ^{\br }_{t}\} $-progressively measurable processes 
$\v =\{ \v _{t}\} _{0\le t\le 1}$ satisfying 
\begin{align*}
 \ex \!\left[ \int _{0}^{1}|\v _{t}|^{2}\,dt\right] <\infty ;
\end{align*}
in order to specify notationally that $\v _{t}$ is a functional of $\br $ 
up to time $t$ and the terminal time is $1$, we denote this set 
by $\A _{1}(\br )$. Let $\ga $ denote the standard Gaussian measure 
on $\R ^{d}$ and $f:\R ^{d}\to \R $ be a measurable function.
Noting that conditions~\thetag{A1} and \thetag{A2} are 
equivalent to both $e^{F}$ and $F$ being in $L^{1}(\wm )$ 
(see \rref{;rtmain1}\thetag{3}), we assume 
\begin{align*}
 \thetag{B}\ \text{$e^{f}\in L^{1}(\ga )$ and $f\in L^{1}(\ga )$}.
\end{align*}
The following is immediate from \tref{;tmain1} applied to 
$F(\br )=f(\br _{1})$: 

\begin{prop}\label{;prestate}
 Under assumption \thetag{B}, we have 
 \begin{align}\label{;vrr}
 \log \ex \!\left[ e^{f(\br _{1})}\right] 
 =\sup _{\v \in \A _{1}(\br )}\ex \!\left[ 
 f\left( \br _{1}+\int _{0}^{1}\v _{t}\,dt\right) 
 -\frac{1}{2}\int _{0}^{1}|\v _{t}|^{2}\,dt
 \right] . 
\end{align}
\end{prop}

Next we recall the exponential hypercontractivity of the 
Ornstein--Uhlenbeck semigroup 
$\ou =\{ \ou _{t}\} _{t\ge 0}$  
defined in the Gaussian space $(\R ^{d},\ga )$. 

For each $t\ge 0$, the operator $\ou _{t}$ acts on $L^{1}(\ga )$ 
in such a way that, for $f\in L^{1}(\ga )$, 
\begin{align*}
 \left( \ou _{t}f\right) \!(x)
 =\int _{\R ^{\D }}
 f\left( e^{-t}x+\sqrt{1-e^{-2t}}y\right) \gss{}(dy), \quad x\in \R ^{\D }. 
\end{align*}
It is well known that $\ou $ enjoys the hypercontractivity, which 
is also known (see \cite[Proposition~4]{be}) to be equivalent to 
the following property that we call the 
{\it exponential hypercontractivity}: for any measurable function 
$f:\R ^{d}\to \R $ satisfying \thetag{B}, 
\begin{align}\label{;eHC}
 \nt{\exp (\ou _{t}f)}{L^{e^{2t}}(\ga )}\le \nt{e^{f}}{L^{1}(\ga )}\quad 
 \text{for all }t\ge 0.
\end{align}
We provide a simple derivation of \eqref{;eHC} 
by means of \pref{;prestate};  formula~\eqref{;vrr} for any 
bounded measurable function $f$ was discovered by Borell \cite{bor} 
independently of Bou\'e--Dupuis \cite{bd} and applied to a simple 
proof of the Pr\'ekopa--Leindler inequality among others.
Our application, which seems to be new to our knowledge, serves 
as another instance of usefulness of the formula, often referred to 
as {\it Borell's formula}, in deriving existing functional inequalities.

Let $f\in L^{1}(\ga )$ and observe the following identity in law 
for every $t\ge 0$: 
\begin{align*}
 \left( 
 \ou _{t}f,\,\ga 
 \right) \stackrel{(d)}{=}
 \left( 
 \ex \!\left[ 
 f(\br _{1})\mid \calF ^{\br }_{e^{-2t}}
 \right] ,\,\pr 
 \right) .
\end{align*}
Indeed, by the independence of $\br _{1}-\br _{e^{-2t}}$ and 
$\br _{e^{-2t}}$, we have, a.s.,
\begin{align*}
 \ex \!\left[ 
 f(\br _{1})\mid \calF ^{\br }_{e^{-2t}}
 \right] 
 =\ex \!\left[ 
 f\!\left( 
 \br _{1}-\br _{e^{-2t}}+x
 \right) 
 \right] \!\big| _{x=\br _{e^{-2t}}}, 
\end{align*}
which has the same law as 
\begin{align*}
 \ex \bigl[ 
 f\bigl( 
 \sqrt{1-e^{-2t}}N_{2}+e^{-t}x
 \bigr) 
 \bigr] \Big| _{x=N_{1}},
\end{align*}
where $N_{1}$ and $N_{2}$ are $d$-dimensional standard Gaussian 
random variables. Therefore the exponential hypercontractivity 
\eqref{;eHC} is equivalently stated as 

\begin{prop}\label{;preHC}
 For every measurable function $f:\R ^{d}\to \R $ satisfying 
 \thetag{B}, it holds that 
 \begin{align}\label{;reHC}
  t\log \ex \!\left[ 
  \exp \left\{ 
  t^{-1}\ex \!\left[ 
  f(\br _{1})\mid \calF ^{\br }_{t}
  \right] 
  \right\} 
  \right] 
  \le \log \ex \!\left[ 
  e^{f(\br _{1})}
  \right] 
 \end{align}
 for all $0<t\le 1$.
\end{prop}

We give a proof of the proposition via \pref{;prestate}. To this end, 
given $f\in L^{1}(\ga )$, we set 
\begin{align*}
 g(t,x):=\ex \!\left[ 
 f(\br _{1}-\br _{t}+x)
 \right] ,\quad 0\le t\le 1,\ x\in \R ^{d}, 
\end{align*}
so that 
\begin{align}\label{;condex}
 \ex \!\left[ 
 f(\br _{1})\mid \calF ^{\br }_{t}
 \right] =g(t,\br _{t})\quad \text{a.s.}
\end{align}
for every $0\le t\le 1$.

\begin{proof}[Proof of \pref{;preHC}]
By appealing to the monotone convergence theorem, it suffices 
to prove \eqref{;reHC} when $f\in L^{1}(\ga )$ is bounded from 
above. Fix $0<t\le 1$ and set 
\begin{align*}
 W_{s}:=\frac{1}{\sqrt{t}}\br _{ts}, && 
 \calF ^{W}_{s}:=\sigma (W_{u},0\le u\le s)\vee \mathcal{N},
\end{align*}
for $0\le s\le 1$, so that $W=\{ W_{s}\} _{0\le s\le 1}$ is a 
standard $d$-dimensional $\{ \calF ^{W}_{s}\} $-Brownian motion.
Note that $\sqrt{t}W_{1}=\br _{t}$ and 
$\calF ^{W}_{1}=\calF ^{\br }_{t}$ by definition. Moreover, 
as $g(t,\br _{t})$ is integrable in view of \eqref{;condex}, 
the function $t^{-1}g\bigl( t,\sqrt{t}x\bigr) ,\,x\in \R ^{d}$, fulfills 
assumption \thetag{B} since we have assumed that $f$ is bounded 
from above. Therefore, noting \eqref{;condex} again, we may 
apply \pref{;prestate} to $t^{-1}g\bigl( t,\sqrt{t}W_{1}\bigr) $ to 
rewrite the left-hand side of \eqref{;reHC} as 
\begin{equation}\label{;rewrite}
\begin{split}
 &t\log \ex \!\left[ 
 \exp \left\{ 
 t^{-1}g\bigl( t,\sqrt{t}W_{1}\bigr) 
 \right\} 
 \right] \\
 &=t\sup _{\v \in \A _{1}(W)}
 \ex \!\left[ 
 t^{-1}g\left( 
 t,\sqrt{t}W_{1}+\sqrt{t}\int _{0}^{1}\v _{s}\,ds
 \right) -\frac{1}{2}\int _{0}^{1}|\v _{s}|^{2}\,ds
 \right] \\
 &=\sup _{\v \in \A _{1}(W)}
 \ex \!\left[ 
 g\left( 
 t,\sqrt{t}W_{1}+\int _{0}^{1}\v _{s}\,ds
 \right) -\frac{1}{2}\int _{0}^{1}|\v _{s}|^{2}\,ds
 \right] \\
 &=\sup _{\v \in \A _{1}(W)}
 \ex \!\left[ 
 f\left( \br _{1}+\int _{0}^{1}\v _{s}\,ds\right) 
 -\frac{1}{2}\int _{0}^{1}|\v _{s}|^{2}\,ds
 \right] .
\end{split}
\end{equation}
Here the second equality follows from the equivalence 
$\sqrt{t}\v \in \A _{1}(W) \iff \v \in \A _{1}(W)$; 
for the third, by recalling the definition of $g$, and 
by noting that the random variables 
\begin{align*}
 \sqrt{t}W_{1}+\int _{0}^{1}\v _{s}\,ds,\quad 
 \int _{0}^{1}|\v _{s}|^{2}\,ds
\end{align*}
are independent of $\br _{1}-\br _{t}$ because they are 
$\calF ^{\br }_{t}$-measurable by the definition of $W$, 
the boundedness of $f$ from above allowed us to 
apply Fubini's theorem. Due to the obvious 
inclusion $\A _{1}(W)\subset \A _{1}(\br )$, the last expression in 
\eqref{;rewrite} is dominated by 
\begin{align*}
 \sup _{\v \in \A _{1}(\br )}\ex \!\left[ 
 f\left( \br _{1}+\int _{0}^{1}\v _{s}\,ds\right) 
 -\frac{1}{2}\int _{0}^{1}|\v _{s}|^{2}\,ds
 \right] ,
\end{align*}
and hence, in virtue of \pref{;prestate} again, by 
$\log \ex \!\left[ e^{f(\br _{1})}\right] $. This proves \eqref{;reHC}.
\end{proof}

\begin{rem}
We may start the proof with bounded measurable functions 
by truncating $f$ as $(f\wedge M)\vee (-N)$ for $M,N>0$. Then 
repeated use of the monotone convergence theorem 
as $N\to \infty $ and then as $M\to \infty $ completes the 
proof. The essential part of the above proof is how Borell's 
formula applies to \eqref{;reHC}.
\end{rem}

By \cite[Proposition~4]{be}, the exponential hypercontractivity 
\eqref{;eHC} is equivalent to the Gaussian logarithmic Sobolev 
inequality in $\R ^{d}$: for any weakly differentiable function $f$
in $L^{2}(\gss{})$ with $|\nabla f|\in L^{2}(\gss{})$, 
\begin{align}\label{;lsi}
 \int _{\R ^{\D }}|f|^{2}\log |f|\,d\gss{}
 \le \nt{|\nabla f|}{L^{2}(\ga )}^{2}
 +\nt{f}{L^{2}(\ga )}^{2}\log \nt{f}{L^{2}(\ga )}; 
\end{align}
we also refer to \cite[Subsection~A.1]{har18} 
in this respect. It is known \cite[Section~3]{bl} that 
the Pr\'ekopa--Leindler inequality implies the logarithmic 
Sobolev inequality; the above exploration provides another path 
from formula~\eqref{;vrr} to \eqref{;lsi}.
\bigskip 

\noindent 
{\bf Acknowledgements.}  The authors are grateful to 
Professor Shigeki Aida for bringing Section~8.1 of \cite{str} 
to their attention as referred to in \rref{;rptmain1}\thetag{2}. 
Their thanks also go to the anonymous referee of \cite{har18}, one 
of whose comments motivated them to do the study in \sref{;sappl}. 
The first author has been supported in part by JSPS KAKENHI 
Grant Number~17K05288.


\end{document}